\documentclass[12pt]{article}

\usepackage{fullpage}

 \usepackage{amsmath,amssymb,amsthm}
    \usepackage[mathscr]{eucal}
\usepackage{pictexwd,dcpic}

\input xy
\xyoption{all}
\xyoption{2cell}
\UseAllTwocells
\input{diagxy}

\title{Finiteness spaces and generalized power series}

\author{\begin{tabular}[t]{c}
        Richard Blute\thanks{Research supported in part by
NSERC Discovery Grants.} \,\,\,\,\,\,\,\,\,\,\,\,\,\,Pierre-Alain Jacqmin$^*$\,\,\,\,\,\,\,\,\,\,\,\,\,\,Philip Scott$^*$\\
        {\small Department of Mathematics and Statistics}\\ [-4pt]
        {\small University of Ottawa}\\  [-4pt]
        {\small Ottawa, Ontario, Canada}\\ [-4pt]
        \\
Robin Cockett$^*$\\
        {\small Department of Computer Science}\\ [-4pt]
        {\small University of Calgary}\\  [-4pt]
        {\small Calgary, Alberta, Canada}\\ [-4pt]
        \end{tabular}}

\newtheorem{theorem}{Theorem}
\newtheorem{thm}{Theorem}[section]
\newtheorem{lemma}[thm]{Lemma}
\newtheorem{proposition}[thm]{Proposition}

\newtheorem{definition}[thm]{Definition}
\newtheorem{example}[thm]{Example}

\newcommand{\ox}{\otimes}

\newcommand{\rarr}{\rightarrow}

\newcommand{\cC}{\mbox{${\cal C}$}}

\renewcommand{\implies}{\Rightarrow}
\newcommand{\cP}{\mbox{${\cal P}$}}
\newcommand{\cU}{\mbox{${\cal U}$}}
\newcommand{\cV}{\mbox{${\cal V}$}}
\newcommand{\cW}{\mbox{${\cal W}$}}
\newcommand{\cT}{\mbox{${\cal T}$}}
\newcommand{\finrel}{\mbox{${\sf FinRel}$}}
\newcommand{\finpf}{\mbox{${\sf FinPf}$}}
\newcommand{\finf}{\mbox{${\sf FinF}$}}
\newcommand{\seq}{\subseteq}
\newcommand{\cF}{\mbox{${\cal F}$}}
\newcommand{\cd}{\xymatrix}
\newcommand{\ZZ}{\mathbb{Z}}
\newcommand{\M}{\mathbb{M}}
\newcommand{\NN}{\mathbb{N}}

\newcommand{\QQ}{\mathbb{Q}}

\newcommand{\X}{\mathbb{X}}
\newcommand{\Y}{\mathbb{Y}}

\begin{document}

\maketitle

\begin{abstract}
We consider Ribenboim's construction of rings of {\it generalized power series}. Ribenboim's construction makes use of 
a special class of partially ordered monoids and a special class of their subsets. While the restrictions he imposes
might seem conceptually unclear, we demonstrate that they are precisely the appropriate conditions to represent such 
monoids as internal monoids in an appropriate category of Ehrhard's {\it finiteness spaces}. Ehrhard introduced  
finiteness spaces as the objects of a categorical model of classical linear logic, where a set is equipped with a class 
of subsets to be thought of as finitary. Morphisms are relations preserving the finitary structure.
The notion of finitary subset allows for a sharper analysis of computational 
structure than is available in the relational model. For example, fixed point operators fail to be finitary. 

In the present work, we take morphisms to be partial functions preserving the finitary structure rather than relations. 
The resulting category is symmetric monoidal closed, complete and cocomplete.
Any pair of an internal monoid in this category and a ring induces a ring of generalized power series 
by an extension of the Ribenboim
construction based on Ehrhard's notion of {\it linearization} of a finiteness space. We thus further generalize Ribenboim's 
constructions. We give several examples of rings which arise from this construction, including the ring of {\it Puiseux series} and the ring of 
{\it formal power series generated by a free monoid.}

\end{abstract}


\section{Introduction}

Rings of power series are objects of fundamental importance in any number of settings in mathematics and 
theoretical computer science. The applications to algebra and analysis are numerous and well-known. In theoretical computer 
science, power series arise for example in the coinductive analysis of streams~\cite{Rut}, as well as in the study of automata and formal language 
theory~\cite{Droste,Eilenberg}.  Thus any framework which generalizes 
and provides a conceptual basis for understanding such rings is of great interest. 

Ribenboim introduced his notion of {\it generalized power series}~\cite{Rib1,Rib2,Rib3} in order to study rings of arithmetic functions.
But the construction is quite general and gives a great many examples, some of which are discussed below. The construction is 
functorial in nature and thus can be analyzed via category theory. Ribenboim begins with a special class of {\it partially ordered monoids}
(pomonoids), which he calls {\em strict} pomonoids. He considers those functions from the pomonoid to a ring such that the 
support (the inverse image of the complement of~0) is artinian and narrow (defined below). He demonstrates that the Dirichlet convolution 
formula lifts to this setting and thus one obtains a ring which can sensibly be thought of as a ring of power series. 

{\it Finiteness spaces} were introduced by Ehrhard~\cite{Ehr2} as an enrichment of the usual relational model of linear logic~\cite{girard}. A finiteness
space is a set equipped with a class of subsets, which are to be thought of as finitary. A morphism between finiteness spaces is a relation
preserving the finitary structure. Ehrhard's model provides for a much finer analysis of the computational structure of linear logic. Fixed point operators
in particular fail to be finitary, as one would expect. While Ehrhard was interested in constructing a model of linear logic and hence
chose relations as his morphisms, in our study of monoids it seems more appropriate to consider (partial) functions preserving
the finitary structure instead. We call such (partial) functions {\it finitary} (partial) functions. It turns out that the category with functions is symmetric monoidal but not closed, complete or cocomplete, while the category with partial functions is complete, cocomplete and symmetric monoidal closed. 

While the conditions that Ribenboim requires in his construction (the assumption that supports must be artinian and narrow and that the pomonoid
must be strict) seem conceptually unclear, they are precisely the assumptions one needs to view these objects as finiteness spaces. In
particular, we show that for any poset, if one defines the finitary subsets to be the artinian and narrow subsets, then the result
is a finiteness space. If one considers the category {\sf StrPos} of posets and strict homomorphisms (i.e., those morphisms that preserve
strict inequality), then this category is (symmetric) monoidal and the internal monoids are precisely the strict pomonoids of Ribenboim. 
Furthermore if one again defines the finitary subsets to be the artinian and narrow subsets, then one obtains an internal monoid  
in the appropriate category of finiteness spaces. We do so by showing that the constructions described above are functorial and monoidal, thus take 
monoids to monoids. 

The final piece of the puzzle is Ehrhard's {\it linearization} of a finiteness space. For a chosen ring, one assigns to a finiteness space the set of all
functions from the space to the ring whose support is finitary. We show that the linearization of an internal monoid is 
a ring and in particular the linearization of the finiteness space associated to a strict pomonoid is precisely Ribenboim's construction. Ehrhard's linearization of finiteness spaces provided one of the first examples of {\it differential categories}~\cite{ER1,diffcat} and in future work, we intend to study these rings from that 
perspective.

Terminology: our rings are supposed to be unitary, but not necessarily commutative.

\section{Ribenboim's generalized power series}

We now review the structure that Ribenboim called {\it generalized power series}, which we will call {\it 
Ribenboim power series}.\footnote{Note that Ribenboim assumes commutativity of both the underlying ring and the 
pomonoid. In fact, neither assumption is necessary and we have modified the definitions accordingly.} The presentation is based on 
those in~\cite{Rib1,Rib2,Rib3}.

Let $(M, \cdot, \leq )$ be a partially ordered monoid (or {\it pomonoid}), i.e., a monoid in the category {\sf Pos} of posets 
and order-preserving maps. We say that $M$ is \emph{strictly ordered} (or is a {\it strict pomonoid}) if
$$s < s' \implies s\cdot t < s' \cdot t  \mbox{ and } t\cdot s<t\cdot s' \,\,\,\, \forall s,s',t \in M \ .$$

A poset is \emph{artinian} if all strictly descending chains are finite; that is, if any 
list $(m_1 > m_2 > \cdots )$ is finite. It is \emph{narrow} if all discrete subsets are finite; 
that is, if any subset of elements mutually unrelated by $\leq$ is finite. It is {\it noetherian}
if every strictly ascending chain is finite. We will use the following result. It was crucial in~\cite{Rib1} in 
proving Proposition~\ref{PropX}.

\begin{lemma} \label{fund}
Let $(P, \leq )$ be an artinian and noetherian poset. Then $P$ is narrow if and only if $P$ is finite.
\end{lemma} 

Since Lemma~\ref{fund} is frequently cited in this field, but a proof is typically not given, we include a proof
as an almost immediate corollary of Ramsey's Theorem for infinite posets. 

\begin{proof}
$(\Leftarrow)$ is obvious.  As for ($\Rightarrow$), suppose $P$ is narrow and infinite (as well as artinian and noetherian).  
By Grillet~\cite{Grillet}, Proposition~B.2.3, the artinian and noetherian conditions are equivalent to saying ``every chain of $P$ is finite" (this 
uses the Axiom of Choice).
By Hodges~\cite{Hodges}, Corollary~11.1.5, as a consequence of Ramsey's Theorem, we obtain: an infinite poset $P$ 
either contains an infinite chain or it contains an infinite antichain (i.e., an infinite discrete subset whose elements are pairwise incomparable). 
Since $P$ is narrow, the latter is impossible.
Hence $P$ must contain an infinite chain, which contradicts Grillet's theorem.  
\end{proof}

\begin{definition} [Ribenboim,\cite{Rib1}] {\em
Let $A$ be an abelian group and $(P, \leq )$ a poset. Recall that the \emph{support} of a function $f \colon P \to A$ is defined by 
$supp(f) = \lbrace p \in P \,|\, f(p) \neq 0 \rbrace$. Define the \emph{space of Ribenboim power series from $P$ 
with coefficients in $A$}, denoted $G(P,A)$, to be the abelian group of functions $f \colon P \to A$ whose support 
is artinian and narrow, with pointwise addition. 
}\end{definition} 

We have now established all of the necessary structure to define Ribenboim's generalized power series. 

\begin{theorem} [Ribenboim,\cite{Rib1}]\label{theorem Ribenboim construction}
If $(M, \cdot, \leq )$ is a strict pomonoid and $R$ a ring, then $G(M,R)$ is also a ring with
\[
(f \cdot g) (m) = \sum_{(m_1,m_2) \in X_m (f,g)} f(m_1) \cdot g(m_2)
\]
where
\[
X_m (f,g) := \lbrace (m_1,m_2) \in M \times M \,|\, m_1 \cdot m_2 = m \text{ and } f(m_1) \neq 0, g(m_2) \neq 0 \rbrace.
\]
The unit is given by the function $e\colon M\rarr R$ where $e(m)=1_R$ if $m=1_M$ and 0 otherwise.
\end{theorem} 

The fact that the multiplication is well-defined follows from:

\begin{proposition} [Ribenboim,\cite{Rib1}]\label{PropX}
The set $X_m (f,g)$ is finite for $f$, $g \in G(M,R)$.
\end{proposition} 

\noindent There are many examples. See the Ribenboim papers for further discussion. 

\begin{itemize}
\item Let $M=\mathbb{N}$ with the standard order. The result is the usual ring of power series with coefficients in $R$.
\item Let $M=\mathbb{Z}$ with the standard order. The result is the ring of Laurent series with coefficients in $R$.
\item Let $M=\mathbb{N}$ with the discrete order. The result is the usual ring of polynomials in $R$.
\item Let $M=\mathbb{Z}$ with the discrete order. The result is the usual ring of Laurent polynomials in $R$.
\item Let $M=\mathbb{N} \backslash \lbrace 0 \rbrace$ with the operation of multiplication, equipped with the usual ordering. Then $G(M,R)$ is the ring of arithmetic functions with values in $R$, and multiplication is Dirichlet's convolution.
\item Let $M = \mathbb{N} \backslash \lbrace 0 \rbrace$ with the operation of multiplication as above, but now equipped with the divisibility ordering; that is, $m_1 \leq m_2 \iff m_1 | m_2$. Then $G(M,R)$ is a proper subring of the ring of arithmetic functions with values in $R$.
\end{itemize}

\section{Finiteness spaces}

\subsection{Basic constructions}

We now introduce Ehrhard's notion of {\it finiteness space}~\cite{Ehr2}.

\begin{definition} \label{findef}{\em
\begin{itemize}
\item Let $X$ be a set and let \cU\ be a set of subsets of $X$, i.e., $\cU\seq\cP(X)$. Define $\cU^\perp$ by:
\[\cU^\perp=\{u'\seq X \,\,|\,\, \mbox{the set $u'\cap u$ is finite for all $u\in\cU$}\}\]
It is immediate to check that one has $\cU \seq \cU^{\perp\perp}$ and $\cU^{\perp\perp\perp} = \cU^\perp$.

\item A {\em finiteness space} is a pair $\mathbb{X} = (X , \cU)$ with $X$ a set and $\cU\seq\cP(X)$ such that 
$\cU^{\perp\perp}=\cU$. We will sometimes denote $X$ by $|\X|$ and $\cU$ by $\cF(\X)$.

\item A {\em morphism} of finiteness spaces $R\colon \X\rarr \Y$ is a relation $R\colon |\X|\rarr|\Y|$
such that the following two conditions hold:
\begin{enumerate}
\item[(1)] For all $u\in\cF(\X)$, we have $uR\in\cF(\Y)$, where $uR=\{y \in |\Y| \,|\, \exists x \in u, xRy\}$. 
\item[(2)] For all $v'\in\cF(\Y)^\perp$, we have $Rv'\in\cF(\X)^\perp$, where $Rv'=\{x \in |\X| \,|\, \exists y \in v', xRy\}$.
\end{enumerate}
\end{itemize}
It is straightforward to verify that this is a category. We denote it \finrel.
}\end{definition} 

\begin{lemma}  [Ehrhard,\cite{Ehr2}]\label{lemma condition (2')}
In the definition of morphism of finiteness spaces, condition~(2) can be replaced with:
\[\emph{\mbox{($2^{\prime}$) For all $b\in |\Y|$, we have $R\{b\}\in\cF(\X)^\perp$.}}\]
\end{lemma} 

\smallskip

\begin{theorem}  [Ehrhard,\cite{Ehr2}]
\finrel\ is a $*$-autonomous category. The tensor $$\X\ox\Y=(|\X\ox \Y|,\cF(\X\ox \Y))$$ is given 
by setting $|\X\ox \Y|=|\X| \times |\Y|$ and
\begin{align*}
\cF(\X\ox \Y)&=\{u\times v \,|\, u\in \cF(\X),v\in\cF(\Y)\}^{\perp\perp}\\
&=\{w \,|\, \exists u\in\cF(\X), \exists v\in\cF(\Y), w \seq u \times v\}.
\end{align*}
The unit for the tensor is $I=(\{*\},\cP(\{*\}))$ and the duality is given by $(|\X|,\cF(\X))^\perp=(|\X|,\cF(\X)^\perp)$.
\end{theorem} 

\subsection{Other choices of morphism}\label{subsection other choices of morphisms}

The choice of morphisms for finiteness spaces was motivated by the desire to have a $*$-autonomous category. 
For examining internal monoids, relations as morphisms seem not to be the right choice. One has two other sensible options which we consider now.

We first define the category \finf. Objects are finiteness spaces and a morphism $f\colon (X,\cU)\rarr (Y,\cV)$ is a function satisfying the same
conditions as in Definition~\ref{findef}. We define \finpf\ in the same way except now morphisms are partial functions satisfying the same
conditions as in Definition~\ref{findef}.

We note that a partial function $f\colon X\to Y$ satisfying (2) of Definition~\ref{findef} automatically satisfies (1). Indeed, 
given $u \in \cU$ and $v' \in \cV^\perp$, if $u \cap f^{-1}(v')$ is finite, then so is $f(u)\cap v'$ in 
view of the surjective restriction of $f$: $(u \cap f^{-1}(v')) \twoheadrightarrow (f(u)\cap v')$. Thus the 
category \finpf\ (respectively \finf) is equivalent to the category having finiteness spaces as objects and partial functions (respectively 
total functions) $f\colon X\to Y$ satisfying $f^{-1}(v)\in \cU$ for each $v \in \cV$ as morphisms $(X,\cU)\to (Y,\cV)$. The equivalence is obtained 
by mapping the finiteness space $(X,\cU)$ to $(X,\cU^\perp)$ and $f\colon (X,\cU)\to (Y,\cV)$ to $f\colon (X,\cU^\perp)\to (Y,\cV^\perp)$. 
This is a `topological' way of viewing these categories, but in order to develop the `classical theory', we are going to work in \finpf\ and \finf.

It is easy to see that \finf\ and \finpf\ are symmetric monoidal categories and the 
inclusions $$\finf\hookrightarrow \finpf \hookrightarrow \finrel$$ are bijective on objects, (strict) symmetric monoidal functors.

The category \finf\ does have one significant problem, it is not monoidal closed. Indeed  the 
functor $$- \ox (\varnothing,\cP(\varnothing)) \colon \finf\rarr\finf$$  does not have a right adjoint (because \finf\ does 
not have a terminal object).

\bigskip

On the other hand, we do have:

\begin{proposition} 
The category \finpf\ is a symmetric monoidal closed category.
\end{proposition} 

\begin{proof}
Let $(X,\cU)$ and $(Y,\cV)$ be two finiteness spaces. We define the finiteness space $[(X,\cU),(Y,\cV)]$ as follows. Let $A$ be the set
$$A= \{f\in \finpf((X,\cU),(Y,\cV)) \,|\, f \text{ is not the empty partial function} \}$$
and let $\cW$ be the set
\begin{align*}
\cW&=\{w \subseteq A \,|\, w \text{ satisfies (4)}\}\\
&= \{w \subseteq A \,|\, w \text{ satisfies (3) and ($4^\prime$)}\} \subseteq \cP(A)
\end{align*}
where conditions~(3), (4) and ($4^\prime$) are defined as follows:
\begin{itemize}
\item[(3)] for each $u \in \cU$, the union $\bigcup_{f \in w} f(u)$ is in $\cV$,
\item[(4)] for each $u \in \cU$ and each $v' \in \cV^\perp$, the set $\{f \in w \,|\, f(u) \cap v' \neq \varnothing\}$ is finite,
\item[($4^\prime$)] for each $u \in \cU$ and each $y \in Y$, the set $\{f \in w \,|\, y \in f(u)\}$ is finite.
\end{itemize}
It is easy to see that condition~(4) implies condition~($4^\prime$). It also implies condition~(3): 
Given $u \in \cU$ and $v' \in \cV^\perp$, let us denote by $\left<u,v'\right>$
the set $$\left<u,v'\right>=\{f \in A \,|\, f(u) \cap v' \neq \varnothing\}.$$
Then, the set
$$\left(\bigcup_{f \in w} f(u)\right) \cap v' = \bigcup_{f \in w} (f(u) \cap v') = \bigcup_{f \in w \cap \left<u,v'\right>} (f(u) \cap v')$$
is finite since $w \cap \left<u,v'\right>$ is and all $f(u) \cap v'$ are. Conversely, the conjunction of conditions~(3) 
and ($4^\prime$) implies condition~(4). Indeed, for $u \in \cU$ and $v' \in \cV^\perp$, the set
$$\{f \in w \,|\, f(u) \cap v' \neq \varnothing\} = 
\bigcup_{y \in v'} \{f \in w \,|\, y \in f(u)\}=\bigcup_{y \in v'\cap \bigcup_{f \in w} f(u)} \{f \in w \,|\, y \in f(u)\}$$
is finite, being a finite union of finite sets.

Let us now prove that $(A,\cW)$ is a finiteness space. We need to show that $\cW^{\perp\perp} \subseteq \cW$. In view of condition~(4), given $u \in \cU$ and $v' \in \cV^\perp$, the set
$\left<u,v'\right>$ belongs to $\cW^\perp$. This means that for $w \in \cW^{\perp\perp}$, the set $w \cap \left<u,v'\right>=\{f \in w \,|\, f(u) \cap v' \neq \varnothing\}$ is finite and $w\in \cW$. We can thus define $[(X,\cU),(Y,\cV)]$ as the finiteness space $(A,\cW)$.

We now define the partial function
$$ev\colon [(X,\cU),(Y,\cV)] \ox (X,\cU) \rightarrow (Y,\cV)$$
by $$ev(f,x)=\begin{cases}f(x) \qquad\quad\,\, \text{if }f(x)\text{ is defined}\\ \text{undefined} \quad \text{if }f(x)\text{ is undefined.}\end{cases}$$
Let us show that this is a morphism in \finpf. For any $v' \in \cV^\perp$, $w \in \cW$ and $u \in \cU$, we must show that 
$ev^{-1}(v') \cap (w \times u)$ is finite.
But this set is $$\bigcup_{f \in w} \{(f,x) \,|\, x \in u \cap f^{-1}(v')\} = \bigcup_{f \in w\cap \left<u,v'\right>} \{(f,x) \,|\, x \in u \cap f^{-1}(v')\}$$
which is finite since $w\cap \left<u,v'\right>$ is and all $u \cap f^{-1}(v')$ are.

Now let $(Z,\cT)$ be a finiteness space and $g \colon (Z,\cT)\ox(X,\cU) \to (Y,\cV)$ a morphism in \finpf. The unique morphism $h\colon (Z,\cT)\to [(X,\cU),(Y,\cV)]$ making the diagram
$$\cd{(Z,\cT)\ox (X,\cU) \ar[rd]^-{g} \ar[d]_-{h \ox (X,\mathcal{U})} \\ [(X,\cU),(Y,\cV)]\ox (X,\cU) \ar[r]_-{ev} & (Y,\cV)}$$
commutative has to be defined via
$$h(z)=\begin{cases}g(z,-) \qquad\quad\, \text{if }g(z,-) \text{ is not the empty partial function}\\ \text{undefined} \qquad \text{if }g(z,-) \text{ is the empty partial function.}\end{cases}$$
It remains to prove $h$ is a well-defined morphism in \finpf. First, let us show that for $z \in Z$, the partial function $g(z,-)$ is a morphism $(X,\cU)\to (Y,\cV)$.
For $u \in \cU$, $g(z,-)(u)=g(\{z\}\times u)$ which is in $\cV$. So $g(z,-)$ satisfies condition~(1). For condition~($2^\prime$), let $y \in Y$ and $u \in \cU$ and notice that
the set $u \cap g(z,-)^{-1}(y)$ is in bijection with the set $(\{z\} \times u) \cap g^{-1}(y)$ which is finite. To conclude the proof, we still have to show that $h\colon (Z,\cT)\to [(X,\cU),(Y,\cV)]$ is also a morphism in \finpf. For condition~(1), we must show that, given $t \in \cT$, $h(t)$ satisfies (3) and ($4^\prime$). Given $u \in \cU$, the set
$$\bigcup_{f \in h(t)} f(u) = \bigcup_{z \in t} g(z,-)(u)=g(t \times u)$$
is in $\cV$, showing condition~(3). For condition~($4^\prime$), let $u \in \cU$ and $y \in Y$. The first projection
$$g^{-1}(y) \cap (t \times u) \twoheadrightarrow \{z \in t \,|\, y \in g(z,-)(u)\}$$
is a surjection and the assignment $z \mapsto g(z,-)$ is a surjection $$\{z \in t \,|\, y \in g(z,-)(u)\} \twoheadrightarrow \{f \in h(t) \,|\, y \in f(u)\}.$$
Since the set $g^{-1}(y) \cap (t \times u)$ is finite, this demonstrates condition~($4^\prime$). It remains now to prove that $h$ satisfies condition~($2^\prime$). Let $f \in A$ and $t \in \cT$. We need to show that $h^{-1}(f) \cap t$ is finite. Since $f$ is not the empty partial function, we can choose $x\in X$ such that $f(x)$ is defined.
Now, we have an injection
$$h^{-1}(f) \cap t=\{z \in t \,|\, g(z,-)=f\} \to \{(z,x) \,|\, z \in t, g(z,x)=f(x)\}=g^{-1}(f(x)) \cap (t \times \{x\})$$
sending $z$ to $(z,x)$. But since $g^{-1}(f(x)) \cap (t \times \{x\})$ is finite, this concludes the proof.
\end{proof}

Notice that the finiteness space $(\varnothing,\cP(\varnothing))$ is a zero object in \finpf\ (and in \finrel). So the empty partial function $X \to Y$ is actually the zero morphism $(X,\cU)\to (Y,\cV)$. The category \finpf\ also has the following additional advantage.

\begin{proposition} 
The pointed category \finpf\ is complete and cocomplete.
\end{proposition} 

\begin{proof}
Let us start showing that \finpf\ has equalisers. Given two parallel morphisms $$\cd{(X,\cU) \ar@<2pt>[r]^-{f} \ar@<-2pt>[r]_-{g} & (Y,\cV)}$$ in \finpf, let us consider the set
\begin{alignat*}{2}
E&= \{x \in X \,|\, f(\{x\})=g(\{x\})\}\\
&= \{x \in X \,|\, \text{either both }f(x) \text{ and }g(x) \text{ are undefined}\\
&\hspace{62pt} \text{or they are both defined and }f(x)=g(x)\}.
\end{alignat*}
Let also $\cW \subseteq \cP(E)$ be $\cW=\{u \in \cU \,|\, u \subseteq E\}$. Then it is routine to show that $$\cW^\perp=\{u' \in \cU^\perp \,|\, u' \subseteq E\},$$
$(E,\cW)$ is a finiteness space and the inclusion $(E,\cW)\hookrightarrow (X,\cU)$ is the equalizer of $f$ and $g$ in \finpf.

Now let $I$ be a set and $(X_i,\cU_i)$ a finiteness space for each $i \in I$. Let us construct the product $\prod_{i \in I} (X_i,\cU_i)$.
For each $i \in I$, we denote by $X'_i$ the disjoint union $X_i \coprod \{\star_i\}$. We consider the product
$$P= \left( \prod_{i \in I} X'_i \right) \setminus \left\{(\star_i)_{i \in I}\right\}$$
and
$$\cW'= \bigcup_{i \in I} \left\{\prod_{j \in I\setminus \{i\}}X'_j \times u'_i \,|\, u'_i \in \cU_i^\perp \right\} \subseteq \cP(P).$$
Then, $(P,\cW'^\perp)$ is a finiteness space and for each $i \in I$, we have a morphism $\pi_i\colon (P,\cW'^\perp)\to (X_i,\cU_i)$ given by
\begin{align*}\pi_i((x'_j)_{j \in I})=\begin{cases}x'_i \qquad &\text{if }x'_i\in X_i\\ \text{undefined} \quad &\text{if }x'_i=\star_i.\end{cases}\end{align*}
This forms the desired product in \finpf. Indeed, let $(Z,\cT)$ be a finiteness space and, for each $i \in I$, $f_i$ be a morphism $(Z,\cT)\to (X_i,\cU_i)$. Then, the unique morphism $g\colon (Z,\cT) \to (P,\cW'^\perp)$ such that $\pi_i g = f_i$ for each $i \in I$ is given by
\begin{align*}
g(z)=\begin{cases}(f'_i(z))_{i \in I} \quad &\text{if there exists }i \in I \text{ such that }z \in Dom(f_i)\\ \text{undefined} \quad &\text{if }f_i(z) \text{ is undefined for all }i \in I\end{cases}
\end{align*}
where $f'_i\colon Z \to X'_i$ is the function defined by
\begin{align*}
f'_i(z)=\begin{cases}f_i(z) \quad &\text{if }z \in Dom(f_i)\\ \star_i \quad &\text{if }z \notin Dom(f_i).\end{cases}
\end{align*}
Let us demonstate that this $g$ indeed satisfies conditions~(1) and ($2^\prime$) for being a morphism in \finpf. For (1), let $t \in \cT$, $i \in I$ and $u'_i \in \cU_i^\perp$.
The set $$g(t) \cap \left(\prod_{j \in I\setminus \{i\}}X'_j \times u'_i \right)=\left\{g(z) \,|\, z \in t \cap {f'_i}^{-1}(u'_i)\right\}=\left\{g(z) \,|\, z \in t \cap f_i^{-1}(u'_i)\right\}$$ is finite since $t \cap f_i^{-1}(u'_i)$ is. This proves that $g(t) \in \cW'^\perp$. For condition~($2^\prime$), let $(x'_j)_{j \in I}$ be an element of $P$. By construction of $P$, there exists $i \in I$ such that $x'_i \in X_i$. Therefore,
$$g^{-1}((x'_j)_{j \in I})\subseteq f_i^{-1}(x'_i) \in \cT^\perp$$
since $f_i$ satisfies ($2^\prime$). Thus $g$ is indeed a morphism in \finpf. This shows that \finpf\ is complete.

We now prove that \finpf\ has coequalisers. Let $f,g\colon (X,\cU)\rightrightarrows (Y,\cV)$ be two morphisms. We first consider the (set-theoretical) quotient
$$Q_1 = Y / R$$ and $q_1\colon Y \twoheadrightarrow Q_1$ the corresponding quotient map where $R$ is the smallest equivalence relation on $Y$ such that $f(x) R g(x)$ for all $x \in Dom(f) \cap Dom(g)$. Then, we consider $Q_2$, the subset of $Q_1$ defined by
$$Q_2=Q_1 \setminus \left(\left\{q_1(f(x)) \,|\, x \in Dom(f) \cap Dom(g)^C\right\} \cup \left\{q_1(g(x)) \,|\, x \in Dom(f)^C \cap Dom(g)\right\}\right)$$
where $Dom(f)^C$ and $Dom(g)^C$ denote as usual the complements in $X$ of $Dom(f)$ and $Dom(g)$ respectively. Finally, we consider $Q_3$, the subset of $Q_2$ defined by
$$Q_3=\left\{a \in Q_2 \,|\, q_1^{-1}(a) \in \cV^\perp \right\}$$ together with the partial (surjective) function $q_3\colon Y\twoheadrightarrow Q_3$ given by
\begin{align*}
q_3(y)=\begin{cases}q_1(y) \quad &\text{if }q_1(y)\in Q_3\\ \text{undefined} \quad &\text{if } q_1(y)\notin Q_3.\end{cases}
\end{align*}
Suppose also that 
$$\cW=\left\{q_3(v) \,|\, v \in \cV\right\} \subseteq \cP(Q_3)$$ 
which induces the finiteness space $(Q_3,\cW^{\perp\perp})$. By construction, we know that $q_3$ gives rise to a
morphism $q_3\colon (Y,\cV) \to (Q_3,\cW^{\perp\perp})$ since it obviously satisfies conditions~(1) and ($2^\prime$). This morphism satisfies $q_3f=q_3g$. Given a morphism $h\colon (Y,\cV) \to (Z,\cT)$ such that $hf=hg$, we can construct a partial function $k \colon Q_3 \to Z$ via
\begin{align*}k(q_3(y))=\begin{cases}h(y) \quad &\text{if } y \in Dom(h)\\ \text{undefined} \quad &\text{if } y \notin Dom(h).\end{cases}\end{align*}
This partial function is well-defined since $R \subseteq R_h$ where $R_h$ is the equivalence relation on $Y$ defined by
\begin{align*}
y R_h y' &\Leftrightarrow h(\{y\})=h(\{y'\})\\ &\Leftrightarrow h(y)=h(y') \text{ (both being defined) or both } h(y) \text{ and } h(y') \text{ are undefined.}
\end{align*}
To prove that $kq_3=h$, the only non-trivial part is to show that for $y \in Dom(h)$, $q_3(y)$ is defined, i.e., $q_1(y) \in Q_3$. If $q_1(y)=q_1(f(x))$ for some $x \in Dom(f) \cap Dom(g)^C$, then $$y R f(x) \Rightarrow y R_h f(x) \Rightarrow f(x) \in Dom(h)$$ which is a contradiction. A similar conclusion holds if $q_1(y)=q_1(g(x))$ for some $x \in Dom(f)^C \cap Dom(g)$. Thus $q_1(y) \in Q_2$. Now, we know that
$$q_1^{-1}(q_1(y)) \subseteq h^{-1}(h(y)) \in \cV^\perp$$
where the first inclusion holds since
$$q_1(y')=q_1(y) \,\,\Rightarrow\,\, y' R y \,\,\Rightarrow\,\, y' R_h y \,\,\Rightarrow\,\, h(y')=h(y).$$
This proves $q_1(y) \in Q_3$ and $kq_3=h$. Moreover, $k$ is the only partial function $Q_3 \to Z$ satisfying this equation. It remains to prove it satisfies condition~(2) for being a morphism $(Q_3,\cW^{\perp\perp}) \to (Z,\cT)$. 
So let $t' \in \cT^\perp$. We have to show that
$$k^{-1}(t')=\{a \in Q_3 \,|\, k(a)\in t'\}=\{q_3(y) \,|\, h(y) \in t'\}=q_3(h^{-1}(t'))$$
is in $\cW^\perp$. Let $v \in \cV$. We obviously have $q_3(h^{-1}(t') \cap v) \subseteq q_3(h^{-1}(t')) \cap q_3(v)$. Conversely, suppose $q_3(y_1)=q_3(y_2)$ with $y_1 \in h^{-1}(t')$ and $y_2 \in v$. This implies $k(q_3(y_1))=h(y_1) \in t'$ and so $k(q_3(y_2))$ is defined and belongs to $t'$. Hence $h(y_2) \in t'$ and $y_2 \in h^{-1}(t')$.
This proves $$q_3(h^{-1}(t')) \cap q_3(v) = q_3(h^{-1}(t') \cap v).$$ Since $h^{-1}(t') \cap v$ is finite, this shows that $q_3(h^{-1}(t')) \in \cW^\perp$.

It now remains to prove the existence of small coproducts in \finpf. Let $I$ be a set and $(X_i,\cU_i)$ be a finiteness space for each $i \in I$.
We consider the disjoint union $\coprod_{i \in I} X_i$ and
$$\cW=\left\{u_{i_1} \amalg \cdots \amalg u_{i_n} \,|\, i_1,\dots,i_n \in I \text{ and } u_{i_k} \in \cU_{i_k} \text{ for each }1\leqslant k \leqslant n\right\}\subseteq \cP\left(\coprod_{i \in I} X_i\right).$$
It is easy to prove that $$\cW^\perp=\left\{\coprod_{i \in I} u'_i \,|\, u'_i \in \cU_i^\perp \text{ for each }i \in I\right\}$$
and $\cW^{\perp\perp}=\cW$. So $\left(\coprod_{i \in I} X_i,\cW\right)$ is a finiteness space. For each $i \in I$, let $$s_i\colon \left(X_i,\cU_i\right) \to \left(\coprod_{j \in I} X_j,\cW\right)$$ be the canonical injection, which is obviously a morphism in \finpf. Given a finiteness space $(Z,\cT)$ with, for each $i \in I$, a morphism $f_i\colon (X_i,\cU_i) \to (Z,\cT)$, we define the partial function $g \colon \coprod_{i \in I} X_i \to Z$ by
\begin{align*}
g(x_i)=\begin{cases}f_i(x_i) \quad &\text{if }x_i \in Dom(f_i)\\ \text{undefined} \quad &\text{if }x_i \notin Dom(f_i)\end{cases}
\end{align*}
for each $x_i \in X_i$. This gives a morphism $g\colon \left(\coprod_{i \in I} X_i,\cW\right) \to (Z,\cT)$ since, for each $t'\in \cT^\perp$,
$$g^{-1}(t')=\coprod_{i \in I} f_i^{-1}(t') \in \cW^\perp.$$
Moreover, we have $gs_i=f_i$ for each $i \in I$ and $g$ is the unique such morphism, proving that $\left(\coprod_{i \in I} X_i,\cW\right)$ is the expected coproduct. So \finpf\ is cocomplete.
\end{proof}

\section{Posets as finiteness spaces, pomonoids as finiteness monoids}

The goal of this section is to explain how we can see a strict pomonoid as a monoid in \finf, and why this is not the case for a general pomonoid. We then generalize Ribenboim's construction to the case of monoids in \finf, and even in \finpf. This will give us a better understanding why Ribenboim needs this strictness assumption when defining the ring $G(M,R)$ of Theorem~\ref{theorem Ribenboim construction}.

\subsection{Posets as finiteness spaces}

\begin{theorem} \label{main}\label{pomfin}
Let $(P,\leq)$ be a poset. Let $\cU$ be the set of artinian and narrow subsets. Then $(P,\cU)$ is a finiteness space.
\end{theorem} 
\begin{proof}
This follows from Lemmas~\ref{lemma U perp = noetherian} and~\ref{lemma V perp = U} below.
\end{proof}

\begin{lemma} \label{lemma U perp = noetherian}
Under the above assumptions, $\cU^\perp$ is the set of noetherian subsets of $P$.
\end{lemma} 
\begin{proof} Let $u'\in\cU^\perp$. Suppose $u'$ is not noetherian. So it has an infinite ascending chain, call this chain $C \seq u'$. $C$ is evidently 
artinian and narrow. So $C\in\cU$. But $C\cap u'=C$ which is infinite.

Conversely, suppose that $u'\seq P$ is noetherian. We must show that for all $u\in\cU$, we have that $u\cap u'$ is finite. This will follow from
Lemma~\ref{fund}.

\begin{itemize}
\item  $u\cap u'$ is narrow and artinian since it is contained in $u$.
\item  $u\cap u'$ is noetherian since it is contained in $u'$. 
\end{itemize}
\end{proof}

\begin{lemma} \label{lemma V perp = U}
Under the above assumptions, if $\cV$ is the set of noetherian subsets of $P$, then $\cV^\perp=\cU$.
\end{lemma} 

\begin{proof}
Let $v'\in\cV^\perp$. Suppose $v'$ is not narrow. So $v'$ has an infinite discrete subset, call it $D$. Note 
that a discrete subset is noetherian and then argue as above. Suppose $v'$ is not artinian. Then it has an infinite
descending chain, which is necessarily noetherian. Again argue as above. This proves $\cV^\perp \seq \cU$. Conversely, notice that $\cU \seq \cU^{\perp\perp}=\cV^\perp$.
\end{proof}

\subsection{Finiteness monoids}

We now want to show that the construction
$$(P,\leq)\mapsto(P,\cU)$$
of Theorem~\ref{pomfin} is functorial. Unfortunately, if we consider it from the usual category {\sf Pos} of posets to any of the 
categories of finiteness spaces we have considered, this is not the case. Indeed, the inverse image under an order-preserving map of a noetherian subset may be not noetherian. However, the problem disappears if we consider strict maps.

\begin{definition} {\em If $(P,\leq)$ and $(Q,\leq)$ are two posets, a map $f\colon P \to Q$ is said to be {\it strict} if $p<p'$ implies $f(p)<f(p')$. In particular, it is a morphism of posets. We denote the category of posets and strict maps by {\sf StrPos}. 
}\end{definition} 

It is now easy to check the following result.

\begin{proposition} \label{proposition E functor}
There is a functor $E \colon {\sf StrPos} \rarr \finf$ defined on objects via the construction of Theorem~\ref{pomfin} and on arrows via $E(f)=f$.
\end{proposition} 

\begin{definition} 
{\em  A {\em finiteness monoid} (respectively a {\em partial finiteness monoid}) is an internal monoid in \finf\ (respectively in \finpf), where we consider the monoidal structures of \finf\ and \finpf\ described in the beginning of Section~\ref{subsection other choices of morphisms}.
}\end{definition} 

We wish to prove that every strict pomonoid induces a finiteness monoid. There is a direct proof of this result, but it is quite grisly. We prefer to use the functorial construction $E \colon {\sf StrPos} \rarr \finf$ of Proposition~\ref{proposition E functor}. For that, we need a further step: We consider in {\sf StrPos} the symmetric monoidal structure where the tensor product is given by the cartesian product in {\sf Pos}. Therefore, the inclusion
$${\sf StrPos} \hookrightarrow {\sf Pos}$$
is a bijective on objects, (strict) symmetric monoidal functor. With that monoidal structure, we can now say that a strict pomonoid is just an internal monoid in {\sf StrPos}. Moreover, we have:

\begin{lemma} 
The functor $E \colon {\sf StrPos} \rarr \finf$ is a strict symmetric monoidal functor.
\end{lemma} 

\begin{proof}
It is obvious that the singleton poset $\{*\}$ is sent to $I=(\{*\},\cP(\{*\}))$. Given two posets $(P,\leq)$ and $(Q,\leq)$, we must show that $E(P)\ox E(Q) = E(P \ox Q)$.
Both of these finiteness spaces have $P \times Q$ as underlying set. The finiteness structure of the former is given by
$$\cF(E(P)\ox E(Q))=\{w \seq u \times v \,|\, u \text{ and }v \text{ are artinian, narrow subsets of }P \text{ and }Q \text{ respectively}\}$$
while the finiteness structure on the latter is given by
$$\cF(E(P \ox Q))=\{w \seq P \times Q \,|\, w \text{ is artinian and narrow}\}.$$
The equality between these two finiteness structures can be proved using the fact that a poset $S$ is artinian and narrow if and only if for each sequence $(s_i)_{i \in \NN}$ in $S$, there exists an infinite sequence $n_1<n_2<n_3<\cdots$ such that $s_{n_1}\leq s_{n_2} \leq s_{n_3} \leq \cdots$. This has been stated without proof in~\cite{Higman}, but can be easily proved via Lemma~\ref{fund}.
\end{proof}

We thus have the following diagram made of strict symmetric monoidal functors:
$$\cd{{\sf StrPos} \,\ar@{^(->}[r] \ar[d]_-{E} & {\sf Pos} \\ \finf \,\ar@{^(->}[r] & \finpf \,\ar@{^(->}[r] & \finrel}$$

Denoting $Mon(\cC)$ for the category of monoids and their morphisms in a monoidal category $\cC$, we then get the following theorem.

\begin{theorem} \label{po-imp-fin}
The functor $E$ induces a functor $Mon(E)\colon Mon({\sf StrPos}) \rarr Mon(\finf)$ from the category of strict pomonoids to the category of finiteness monoids.
\end{theorem} 

\begin{proof}
This is an immediate consequence of the general fact that (lax-)monoidal functors take monoids to monoids.
\end{proof}

\subsection{Linearizing finiteness spaces and generalizing the Ribenboim construction}

Let $A$ be an abelian group and $\X=(X,\cU)$ a finiteness space. Ehrhard defined in~\cite{Ehr2} the abelian group $A\langle \X\rangle$ as the set
$$A\langle \X\rangle=\{f\colon X\rarr A \,|\, supp(f)\in \cU\}$$
together with pointwise addition. Evidently in the case of a poset $(P,\leq)$ with its finiteness structure as determined by Theorem~\ref{pomfin}, we recover $G(P,A)$.
With this in mind, Ribenboim's construction can now be generalized further. We use in the following theorem a partial finiteness monoid and not a finiteness monoid for two reasons. Firstly, this is more general, bringing in Example~\ref{example polynomials of degree at most n}. But the main reason is that the category \finpf, as opposed to \finf, is symmetric monoidal closed, complete and cocomplete, which will turn out to be important properties for the study of Morita theory in future work. 

\begin{theorem} 
If $(\M,\mu\colon \M\ox \M\rarr \M,\eta\colon I\rarr \M)$ is a partial finiteness monoid and $R$ a ring (not necessarily commutative, but with unit), then $R\langle \M\rangle$ canonically has the structure of a ring.
\end{theorem} 

\begin{proof}
Let us denote $\M$ by $(M,\cU)$. First, notice that either $M$ is the empty set or $\eta(*)$ is defined. The multiplication in $R\langle \M\rangle$ is given by
$$(f \cdot g) (m) = \sum_{(m_1,m_2) \in X_m (f,g)} f(m_1) \cdot g(m_2)$$
where
$$X_m (f,g) := \lbrace (m_1,m_2) \in M \times M \,|\, \mu(m_1,m_2) = m \text{ and } f(m_1) \neq 0, g(m_2) \neq 0 \rbrace.$$
The fact that $X_m (f,g)$ is finite simply comes from the fact that the multiplication $$\mu\colon \M\ox \M\rarr \M$$ satisfies condition ($2^\prime$) of Lemma~\ref{lemma condition (2')}. Moreover, $f \cdot g \in R\langle \M\rangle$ since
$$supp(f\cdot g) \seq \mu(supp(f) \times supp(g)) \in \cU$$ using the fact that $\mu$ satisfies condition (1) of Definition~\ref{findef}.
The unit of $R\langle \M\rangle$ is given by the function $e\colon M\rarr R$ where $e(m)=1_R$ if $m=\eta(*)$ and 0 otherwise. The calculation of the ring axioms is straightforward.
\end{proof}

In the case where $\M$ is $Mon(E)(M)$ for a strict pomonoid $M$, we recover the ring $G(M,R)$. Thus we can view the ring associated to an arbitrary partial finiteness monoid as a generalized Ribenboim power series ring.

\section{Examples}

This new approach to generalizing the Ribenboim construction gives many additional interesting examples.

\begin{example} [Puiseux series]{\em 
 A {\em Puiseux series}~\cite{newton,puiseux} with coefficients in the ring $R$ is a series (with indeterminate $T$) of the form
$$\sum_{i \geqslant a}^{+ \infty} r_i T^{i/n}$$
for some integer $a \in \ZZ$, some positive integer $n \in \NN\setminus\{0\}$ and where $r_i \in R$. With the usual sum and product law, they form the ring of Puiseux series with coefficients in $R$. We can see this ring as an example of the above construction as follows. For $a \in \ZZ$ and $n \in \NN\setminus\{0\}$, we consider the following subset of rational numbers
$$u_{a,n}=\left\{\frac{i}{n} \,|\, i \in \ZZ, i \geqslant a\right\} \subset \QQ.$$
Then we define $\cU$ as the down-closure of $\left\{u_{a,n} \,|\, a\in \ZZ, n \in \NN\setminus\{0\} \right\}$ in $\cP(\QQ)$, i.e.,
$$\cU=\, \downarrow \left\{u_{a,n} \,|\, a\in \ZZ, n \in \NN\setminus\{0\} \right\} \subset \cP(\QQ).$$
Let us prove $(\QQ,\cU)$ forms a finiteness space. Let $u \in \cU^{\perp\perp}$. Suppose there exists an infinite sequence of rational numbers (written in irreducible form) in $u$
$$\frac{a_1}{b_1},\frac{a_2}{b_2},\frac{a_3}{b_3},\dots$$
such that $0<b_1<b_2<b_3<\cdots$. Then, using Proposition~1 in~\cite{Ehr2}, there exists an infinite subsequence of this sequence whose elements belong to some common $u_{a,n}$. But this is clearly impossible. So there exists $n \in \NN\setminus\{0\}$ such that any element in $u$ can be written as $\frac{i}{n}$ for some $i \in \ZZ$. Now, suppose there exists an infinite sequence of elements in $u$
$$\frac{i_1}{n},\frac{i_2}{n},\frac{i_3}{n},\dots$$
such that $i_1>i_2>i_3>\cdots$. Then, using again Proposition~1 in~\cite{Ehr2}, we can deduce the existence of an infinite subsequence of this sequence whose elements belong to some common $u_{a,m}$. Since this is impossible, such a sequence does not exist and we know there exists an $a \in \ZZ$ such that $u \subseteq u_{a,n}$, proving that $\cU^{\perp\perp}\subseteq \cU$ and so $(\QQ,\cU)$ is a finiteness space.

Next, we want to show that $(\QQ,\cU)$ equipped with the classical $+$ and $0$ is a monoid in \finpf\ (actually, even in \finf). The only non-trivial fact is that
$$+\colon (\QQ,\cU)\ox (\QQ,\cU) \to (\QQ,\cU)$$ satisfies conditions~(1) and ($2^\prime$). Let $a,b \in \ZZ$ and $n,m \in \NN\setminus\{0\}$. It is easy to see that
$$+(u_{a,n}\times u_{b,m}) \subseteq u_{am+bn,nm}$$
proving condition~(1). For condition~($2^\prime$), let $\frac{c}{p}$ be a rational number with $c \in \ZZ$ and $p \in \NN\setminus\{0\}$. We need to show that
$$+^{-1}\left(\frac{c}{p}\right)\cap (u_{a,n} \times u_{b,m})=\left\{\left(\frac{i}{n},\frac{j}{m}\right) \,|\, i \geqslant a, j \geqslant b \text{ and }imp+jnp=nmc\right\}$$
is a finite set. For each $\left(\frac{i}{n},\frac{j}{m}\right)$ in the above set, we have $imp=nmc-jnp \leqslant nmc-bnp$ and so
$$a \leqslant i \leqslant \frac{nmc-bnp}{mp}.$$
So, $i$ can only take a finite number of values. But for each such $i$, there is at most one corresponding $j$, proving there are only finitely many such $\left(\frac{i}{n},\frac{j}{m}\right)$.

Since $((\QQ,\cU),+,0)$ is a monoid in \finpf, we can consider the ring $R\langle(\QQ,\cU)\rangle$, which is nothing but the ring of Puiseux series with coefficients in $R$.}
\end{example} 

\begin{example} [Formal power series] {\em 
Let $A$ be a set (called in this case the \textit{alphabet}). Then, let $M$ be the free monoid generated by $A$. The finiteness space $(M,\cP(M))$ has a monoid structure in \finpf\ (and actually even in \finf) given by the classical monoid structure of $M$. The only non-trivial part here, is to check that the multiplication
$$\cdot\,\colon (M,\cP(M)) \ox (M,\cP(M)) \to (M,\cP(M))$$ satisfies condition~($2^\prime$).
This is due to the fact that, since $M$ is freely generated by $A$, for each $m \in M$, there are only finitely many $(m_1,m_2) \in M^2$ such that $m_1\cdot m_2=m$.
Then the ring $R\langle(M,\cP(M))\rangle$ is called the ring of formal power series with exponents in $M$ and coefficients in $R$ and is constructed as the set of all maps $M \to R$, together with the classical sum and product of formal power series.
}\end{example} 

\begin{example} [Polynomials of degree at most $n$]\label{example polynomials of degree at most n} {\em 
Let $n$ be a natural number and $X=\{0,\dots,n\}$. The finiteness space $(X,\cP(X))$ has a monoid structure $((X,\cP(X)),\mu,\eta)$ in \finpf:
$$\eta \colon (\{*\},\cP(\{*\})) \to (X,\cP(X))$$
maps $*$ to $0$ and
$$\mu \colon (X,\cP(X))\ox (X,\cP(X))=(X \times X,\cP(X \times X)) \to (X,\cP(X))$$
is defined by
\begin{align*}
\mu(a,b)=\begin{cases}a+b \quad &\text{if }a+b \leqslant n\\ \text{undefined} \quad &\text{if }a+b>n.\end{cases}
\end{align*}
The corresponding ring $R\langle(X,\cP(X))\rangle$ is then nothing else than $R_{\leqslant n}[T]$, the ring of polynomials of degree at most $n$ and coefficients in $R$. The multiplication is generated by
\begin{align*}
(r_1T^a)\cdot (r_2T^b)=\begin{cases}r_1r_2T^{a+b} \quad &\text{if }a+b \leqslant n\\ 0 \quad &\text{if }a+b>n.\end{cases}
\end{align*}
}\end{example} 

\section{Future work}

Differentiation provides important operators on power series rings and a natural question is whether one can differentiate the generalized 
power series that arise in this paper. Indeed, in the commutative case, the category of linearized finiteness spaces provided one of the first examples of {\it 
differential categories}~\cite{diffcat}, used in the study of models of differential linear logic~\cite{ER1}. It will be of interest to study differentiation of these 
generalized series and the extent to which they fit into the differential category framework. 

Laurent series are of great interest for any number of reasons, but one place they arise is in renormalization in quantum field theory~\cite{Marc}. 
This ring has a {\it Rota-Baxter operator}~\cite{Guo} which is used in the Connes-Kreimer approach to renormalization~\cite{ebrahimifard}. 
Guo and Liu~\cite{GL} subsequently studied when a projection operator on Ribenboim power series is in fact a Rota-Baxter operator. A similar characterization of 
this operator and its functorial properties in the context of finiteness monoids is an ongoing project.

Finally we mention {\it Morita theory}~\cite{AF}. Two rings are {\it Morita equivalent} if their categories of representations are equivalent. This theory
generalizes to any number of settings. For example, the Morita theory of pomonoids~\cite{Laan2,Laan} is a well-established field. It is of great
interest to determine the extent to which the functorial constructions presented here relate Morita theory for pomonoids and partial finiteness monoids to Morita theory
for rings. 

\bibliographystyle{plain}

\begin{thebibliography}{10}
\bibitem{AF} F. Anderson and K. Fuller, Rings and categories of modules, \emph{Springer-Verlag}, (1992). 

\bibitem{diffcat} R. Blute, J.R.B. Cockett and R.A.G. Seely, Differential categories, \emph{Mathematical Structures in Computer Science} \textbf{16}, pp. 1049--1083, (2006).

\bibitem{Droste} M. Droste and  W. Kuich, Semirings and formal power series, \emph{Handbook of Weighted Automata, Springer}, pp. 3--28, (2009).

\bibitem{ebrahimifard} K. Ebrahimi-Fard and L. Guo, Rota-Baxter algebras in renormalization of perturbative quantum field theory, \emph{Fields Inst. Commun.} \textbf{50}, pp. 47--105, (2007). 

\bibitem{Ehr2} T. Ehrhard, Finiteness spaces, \emph{Mathematical Structures in Computer Science} \textbf{15}, pp. 615--646, (2005).

\bibitem{ER1} T. Ehrhard and L. Regnier, The differential lambda-calculus, \emph{Theoretical Computer Science} \textbf{309}, pp. 1--41, (2003).

\bibitem{Eilenberg} S. Eilenberg, Automata, languages, and machines, Vol. A., \emph{Academic Press}, (1974).

\bibitem{girard} J.-Y. Girard, Linear logic, \emph{Theoretical Computer Science} \textbf{50}, pp. 1--102, (1987).

\bibitem{Grillet} P.A. Grillet, Algebra, \emph{John Wiley \& Sons}, (1999).
 
\bibitem{Guo} Li Guo, Surveys of modern mathematics volume IV: An introduction to Rota-Baxter algebra, \emph{International Press}, (2010).

\bibitem{GL} L. Guo and Z. Liu, Rota-Baxter operators on generalized power series rings, \emph{J. Algebra and Its Applications} \textbf{8}, pp.  557--564, (2009).

\bibitem{Higman} G. Higman, Ordering by divisibility in abstract algebras, \emph{Proc. London Math. Soc. (3)} \textbf{2}, pp. 326--336, (1952).

\bibitem{Hodges} W. Hodges, Model theory, \emph{Camb. Univ. Press}, (2008).

\bibitem{Laan2} V. Laan, Morita theorems for partially ordered monoids, \emph{Proceedings of the Estonian Academy of Sciences} \textbf{60}, pp. 221--237, (2011).

\bibitem{Laan} V. Laan, On Morita equivalence of partially ordered monoids, \emph{Applied Categorical Structures} \textbf{22}, pp. 137--146, (2014). 

\bibitem{Marc} M. Marcolli, Feynman motives, \emph{World Scientific}, (2009).

\bibitem{newton} I. Newton, The method of fluxions and infinite series; with its application to the geometry of curve-lines, (1671), translated from Latin by J. Colson, \emph{London: Henry Woodfall}, (1736).

\bibitem{puiseux} V. Puiseux, Recherches sur les fonctions alg\'ebriques, \emph{J. Math. Pures Appl. (1)} \textbf{15}, pp. 365--480, (1850).

\bibitem{Rib1} P. Ribenboim, Rings of generalized power series: Nilpotent elements, \emph{Abh. Math. Semin. Univ. Hambg.} \textbf{61}, pp. 15--3, (1991).

\bibitem{Rib2} P. Ribenboim, Noetherian rings of generalized power series, \emph{J. Pure Appl. Algebra} \textbf{79}, pp. 293--312, (1992). 

\bibitem{Rib3} P. Ribenboim, Rings of generalized power series II: Units and zero-divisors, \emph{Journal of Algebra} \textbf{168}, pp. 71--89, (1994).

\bibitem{Rut} J. Rutten, Behavioural differential equations: A coinductive calculus of streams, automata, and power series, \emph{Theoretical Computer Science} \textbf{308}, pp. 1--53, (2003). 
\end{thebibliography}

\end{document}